\documentclass{article}
\usepackage[utf8]{inputenc}
\usepackage{amsmath}
\usepackage{amsthm}
\usepackage{amssymb}
\usepackage{mathabx}
\usepackage{xcolor}
\usepackage{bm}
\usepackage{graphicx}
\usepackage{bbm}
\usepackage{subcaption}
\usepackage{wasysym}
\usepackage{hyperref}
\usepackage[linesnumbered,ruled]{algorithm2e}
\usepackage{adjustbox}
\usepackage{mathbbol}
\usepackage{authblk}

\usepackage{subcaption}

\def \x{\textbf{x}}
\def \y{\textbf{y}}

\def \PP{\mathcal{P}}
\def \II{\mathcal{I}}

\def \FF{\mathcal{F}}

\def \probX{\PP(X)}

\def \probY{\PP(Y)}

\def \PPXy{\PP(X\times Y_1)}

\def \FFmn{\FF(\mu,\nu)}
\def \Pimn{\Pi(\mu,\nu)}
\def \IImn{\II(\mu,\nu)}

\def \Xfu{X\times Y_1}
\def \Xfd{X_2\times Y}

\def \rpos{[0,\infty)}
\def \erre{\mathbb{R}}

\def \intxY{\int_{X_2\times Y}}
\def \intXy{\int_{X\times Y_1}}

\def \proj#1{(\mathfrak{p}_{#1})_\#}
\def \projX{(\mathfrak{p}_X)_\#}
\def \projY{(\mathfrak{p}_Y)_\#}

\def \projxx{(\mathfrak{p}_{X_2})_\#}

\def \projyy{(\mathfrak{p}_{Y_1})_\#}

\def \projXy{(\mathfrak{p}_{X\times Y_1})_\#}
\def \projxY{(\mathfrak{p}_{X_2\times Y})_\#}
\def \projxy{(\mathfrak{p}_{Y_1\times X_2})_\#}

\def \pr#1{(\mathfrak{p}_{#1})}

\def \prxy{(\mathfrak{p}_{X_2\times Y_1})}

\def \infpi{\inf_{\pi \in \Pimn}}

\def \infcf{\inf_{\ff \in \FFmn}}

\def \infIp{\inf_{\sigma \in \II(\mu,\nu)}}

\def \mn{(\mu,\nu)}
\def \ff{(\funo,\fdue)}

\def \funo{f^{(1)}}
\def \fdue{f^{(2)}}

\def \sigu{\sigma_{|x_2}^{(1)}}
\def \sigd{\sigma_{|y_1}^{(2)}}

\def \T{\mathbb{T}_c}

\def \ST{\mathbb{CT}_c}
\def \Z{\mathbb{Z}_c}

\def \erreinf{\erre\cup\{+\infty\}}
\def \Leb{\mathcal{L}}

\newtheorem{theorem}{Theorem}

\newtheorem{corollary}{Corollary}
\newtheorem{lemma}{Lemma}
\newtheorem{remark}{Remark}

\newtheorem{example}{Example}
\newtheorem{definition}{Definition}

\title{On the Pythagorean Structure of the Optimal Transport for Separable Cost Functions}
\author{G. Auricchio}

\begin{document}

\maketitle

\begin{abstract}
In this paper, we study the optimal transport problem induced by separable cost functions. In this framework, transportation can be expressed as the composition of two lower-dimensional movements. Through this reformulation, we prove that the random variable inducing the optimal transportation plan enjoys a conditional independence property. We conclude the paper by focusing on some significant settings. 
In particular, we study the problem in the Euclidean space endowed with the squared Euclidean distance. In this instance, we retrieve an explicit formula for the optimal transportation plan between any couple of measures as long as one of them is supported on a straight line.

\end{abstract}

\vspace{1cm}
\small{\textbf{Keywords}: Wasserstein distance, Optimal Transport, Structure of the Optimal Plan\\

\textbf{AMS}: 49Q10, 49Q20, 49Q22}

\section{Introduction}

The Optimal Transport (OT) problem is a classical minimization problem dating back to the work of Monge \cite{Monge1781} and Kantorovich \cite{kantorovich2006translocation,kantorovich1960mathematical}. In this problem, we are given two probability measures, namely $\mu$ and $\nu$, and we search for the cheapest way to reshape $\mu$ into $\nu$. The effort needed in order to perform this transformation depends on a cost function, which describes the underlying geometry of the product space of the support of the two measures. In the right setting, this effort induces a distance between probability measures.

During the last century, the OT problem has been fruitfully used in many applied fields such as the study of systems of particles by Dobrushin \cite{dobrushin1989dynamical}, the Boltzmann equation by Tanaka \cite{Tanaka1978,Tanaka1973,Murata1974}, and the field of fluidodynamics by Yann Brenier \cite{Brenier1991}. All these results pointed out that, by a qualitative description of optimal transport, it was possible to gain insightful information on many open problems. For this reason, the OT problem has become a topic of major interest for analysts, probabilists and statisticians \cite{Villani2008,AGS,Santambrogio2015}. In particular, a plethora of results concerning the uniqueness \cite{uniqueness,10.2307/827090,figalli2007existence}, the structure \cite{struct,CUESTAALBERTOS199786,RePEc:eee:jmvana:v:32:y:1990:i:1:p:48-54}, and the regularity \cite{loeper2009regularity,bouchitte07} of the optimal transportation plan in the continuous framework has been proved.


In this paper, we specialize the problem to the separable cost functions. A cost function is said to be separable if it is the sum of two independent pieces. On $\erre^2$, this means 
\begin{eqnarray*}
c:\erre^2\times\erre^2\to\erre,\quad \quad\quad \quad c(\x,\y)=c_1(x_1,y_1)+c_2(x_2,y_2),
\end{eqnarray*}
where $\x=(x_1,x_2)$ and $\y=(y_1,y_2)$. The most famous ground distances meeting this condition are the cost functions induced by $l^p-$norms, which are
\begin{equation}
    \label{introd_c_pdef}
c_p(\x,\y)=(|x_1-y_1|)^p+(|x_2-y_2|)^p.
\end{equation}

\noindent In \cite{Auricchio2018}, the authors exploited the geometry induced by this class of cost functions to reformulate the transportation problem between discrete measures as an efficient uncapacitated minimum cost flow problem. In this paper, we delve further into the properties of those flows, which we will be calling \emph{cardinal flows}, to retrieve significant information on the Wasserstein cost and the structure of optimal transportation plans. Let $\ff$ be an optimal cardinal flow between $\mu$ and $\nu$ for the generic separable cost function $c=c_1+c_2$ and let $\zeta$ be the common marginal between $\funo$ and $\fdue$, i.e.
\[
\zeta=\pr{Y_1\times X_2}_\# \funo= \pr{Y_1\times X_2}_\# \fdue.
\]

In our first and main result, we show that the restriction of $\funo$ to a horizontal line is an optimal transportation plan between the restriction of $\mu$ and $\zeta$ on the same line. Similarly, the restriction of $\fdue$ on any vertical line is an optimal transportation plan between the restriction of $\zeta$ and $\nu$. 
By expressing this property through the conditional laws of $\mu$, $\zeta$, and $\nu$, we find 
\begin{equation}
\label{introd_formula1}
   W_c(\mu,\nu)=\int_{\erre}W_{c_1}(\mu_{|x_2},\zeta_{|x_2})d\mu_2+\int_{\erre} W_{c_2}(\zeta_{|y_1},\nu_{|y_1})d\nu_1,
\end{equation}
where $W_c\mn$ is the Wasserstein cost between $\mu$ and $\nu$. We call a measure $\zeta$ satisfying the identity \eqref{introd_formula1} a \emph{pivot measure}. We show that, given an optimal cardinal flow and its pivot measure, it is possible to retrieve an optimal transportation plan. Moreover, among all the possible transportation plans, there exists one whose first and last marginals are independent given the other two.

We conclude the paper by analyzing formula \eqref{introd_formula1} in two specifics frameworks.

In the first one the cost function is the sum of two independent distances, i.e. $c=d$ is a separable distance. In this setting, the Wasserstein distance $W_d$ inherits a weaker version of the separability, in particular, we have
\[
W_d\mn=W_d(\mu,\zeta)+W_d(\zeta,\nu).
\]
for any pivot measure $\zeta$.

The other case we consider is $X=Y=\erre^2$ and $c$ is a cost function as in \eqref{introd_c_pdef}. In this case, knowing the pivot measure is enough to retrieve an optimal transportation plan. This is possible because the optimal transportation plan between one-dimensional measures has an explicit formula. In particular, we can rewrite formula \eqref{introd_formula1} through the pseudo-inverse cumulative functions of the conditional laws of $\mu$, $\nu$, and $\zeta$ and find
\begin{align*}
    W_c\mn=   \min_{\zeta} &\int_{\erre}\int_{[0,1]} (|F^{[-1]}_{\mu_{|x_2}}(s)-F^{[-1]}_{\zeta_{|x_2}}(s)|)^pdsd\mu_2\\
    &\;+\int_{\erre}\int_{[0,1]} (|F^{[-1]}_{\zeta_{|y_1}}(t)-F^{[-1]}_{\nu_{|y_1}}(t)|)^pdtd\nu_1.
\end{align*}
Finally, we show that, when the cost function is the squared Euclidean distance, this formula allows us to express the optimal cardinal flow (and therefore the optimal transportation plan) between a measure supported on a straight line and any other measure through a closed formula.
\section{Preliminaries}

\noindent In this section, we fix our notation and recall the Optimal Transportation problem. To keep the discussion as general as possible, we only require $X$ and $Y$ to be Polish spaces.

\noindent For a complete discussion on these topics, we refer to \cite{AGS,bogachev2007measure,Villani2008}.

\subsection{Basic Notions of Measure Theory}
\noindent Given a Polish space $(X,d)$, we denote with $\PP(X)$ the set of all the Borel probability measures over $X$. Given $\mu\in\PP(X)$, we denote with $L^p_\mu$ the set of $L^p$ integrable functions with respect to the measure $\mu$. Given $\mu\in\PP(X)$ and $\nu\in\PP(Y)$, we denote with $\mu\otimes\nu$ the direct product measure of $\mu$ and $\nu$. We say that $\mu\in\PP(X)$ has finite $p-$th moment if, given any $x_0\in X$,
\[
\int_{X}d(x,x_0)^pd\mu<+\infty.
\]
Finally, we denote with $P_p(X)$ the set of measures with finite $p-$th moment.

\begin{definition}[Push-forward of measures]
Let $X$ and $Y$ be two Polish spaces, $T:X\to Y$ a measurable function, and $\mu\in \PP(X)$. The push-forward measure of $\mu$ through $T$ is defined as
\[
T_{\#}\mu(B):=\mu(T^{-1}(B))
\]
for each Borel set $B\subset Y$. 
\end{definition}

\begin{remark}
Given a measurable function $T:X\to Y$ and $\mu\in\PP(X)$, there is an integral equation that characterize the push-forward measure $T_\#\mu$. Given any  measurable function $\phi$ on $Y$ the push-forward measure of $\mu$ through $T$ is the only probability measure $\nu$ satisfying the identity
\begin{equation}
    \label{eq:pushforwardchar}
    \int_X \phi\circ T d\mu =\int_Y \phi d\nu.
\end{equation}
\end{remark}

\begin{lemma}[Chain Rule for Push-forwards]
\label{lm:chainrule}
Let $X$, $Y$, and $Z$ be three Polish spaces and let $T_1:X\to Y$, $T_2:Y\to Z$ be two measurable functions. Given any $\mu \in \PP(X)$, it holds true the chain rule
\[
(T_2\circ T_1)_\# \mu = (T_2)_\#(T_1)_\#\mu.
\]
\end{lemma}

\noindent Let us assume that the Polish space $X$ is the direct product of two Polish spaces $X_1$ and $X_2$. The projections over $X_1$ and $X_2$ are then defined as 
\[
\pr{X_1}(\x):=x_1 \quad \text{and} \quad \pr{X_2}(\x):=x_2
\]
respectively, where $\x=(x_1,x_2)$ is a generic point of $X$. Those functions are continuous (and hence measurable).

\begin{definition}[Marginal Probabilities]
Let $X = X_1\times  X_2$ be a Polish space and $\mu \in \PP(X)$. The first marginal of $\mu$ is the probability measure $\mu_1\in\PP(X_1)$ defined as
\[
\mu_1:=\proj{X_1}\mu.
\]
Similarly, the second marginal of $\mu$ is the probability measure $\mu_2\in\PP(X_2)$ defined as
\[
\mu_2:=\proj{X_2}\mu.
\]
\end{definition}

\begin{lemma}[Gluing Lemma]
\label{lm:gluing}
For $i=1,2,3$, let $X_i$ be Polish spaces and $\mu_i \in \PP(X_i)$. Moreover, let us take $\pi_1 \in \PP(X_1\times X_2)$  such that
\[
\proj{X_1}\pi_1=\mu_1\quad\quad \text{and}\quad\quad\proj{X_2}\pi_1=\mu_2
\]
and $\pi_2 \in \PP(X_2\times X_3)$ such that
\[
\proj{X_2}\pi_2=\mu_2\quad\quad \text{and}\quad\quad\proj{X_3}\pi_2=\mu_3.
\]
\noindent Then, there exists $\pi \in \PP(X_1\times X_2 \times X_3)$ such that
\[
\proj{X_1\times X_2}\pi=\pi_1\quad\quad \text{and}\quad \quad\proj{X_2\times X_3}\pi=\pi_2.
\]
\end{lemma}

\begin{definition}[Disintegration of a Measure]
\label{th:disintegration}
Let $f:X \to Y$ be a measurable function and $\mu \in \PP(X).$ We say that a family $\{\mu_y\}_{y\in Y}$ is a disintegration of $\mu$ according to $f$ if every $\mu_y$ is a probability measure concentrated on $f^{-1}(\{y\})$ such that
\[
y \to \int_{X}\phi d\mu_y
\]
is a measurable function and 
\begin{equation}
    \label{eq:disintegrationformula}
\int_X\phi d\mu=\int_Y\bigg(\int_{X}\phi d\mu_y\bigg)d(f_\#\mu)
\end{equation}
for every $\phi \in C(X)$. With a slight abuse of notation, the disintegration of $\mu$ is also written as 
\begin{equation}
    \label{eq:disnitegration}
    \mu=\mu_y\otimes \nu
\end{equation}
where $\nu=f_\#\mu.$
\end{definition}

\noindent Whenever $f$ is measurable and $\mu\in\PP(X)$, the disintegration of $\mu$ with respect to $f$ exists (see \cite{bogachev2007measure}, Chapter $10$).  In particular, if $X=X_1\times X_2$ and $f=\pr{X_1}$, formula \eqref{eq:disnitegration} reads as
\[
\mu=\mu_{|x_1}\otimes\mu_1.
\]
The measure $\mu_{|x_1}$ is called conditional law of $\mu$ given $x_1$.


\subsection{The Optimal Transport Problem}

The first formulation of the transportation problem is the one due to Monge and, in modern language, to Kantorovich. In \cite{kantorovich_monge}, the author modelized the transshipment of mass through a probability measure over the product space $X\times Y$. He called those measures transportation plans.

\begin{definition}[Transportation Plan]
Let $\mu$ and $\nu$ be two measures over two Polish spaces $X$ and $Y$. The probability measure $\pi \in \PP(X\times Y)$ is a transportation plan between $\mu$ and $\nu$ if
\[
\proj{X}\pi=\mu\quad \quad \text{and} \quad \quad \proj{Y}\pi=\nu.
\]
We denote with $\Pimn$ the set of all the transportation plans between $\mu$ and $\nu.$
\end{definition}

\begin{definition}[Transportation Functional]
\label{def:transp_funct}
Let $\mu\in\PP(X)$, $\nu\in\PP(Y)$, and let $c:X\times Y\to \erreinf$ be a lower semi-continuous function such that there exist two upper semi-continuous functions $a\in L^{1}_\mu$ and $b\in L^{1}_\nu$ such that
\begin{equation}
\label{eq:condoncost2}
    c(x,y)\geq a(x)+b(x)
\end{equation}
for each $(x,y) \in X\times Y$. The transportation functional $\T:\Pimn\to \erre\cup\{+\infty\}$ is defined as
\begin{equation}
    \label{eq:transportfunct}
    \T(\pi):=\int_{X\times Y}cd\pi.
\end{equation}
\end{definition}

\noindent The conditions asked to the cost function in Definition \ref{def:transp_funct} are the minimal ones for which it makes sense defining the integral in \eqref{eq:transportfunct}. Ensured those conditions, we define the following minimum problem.

\begin{definition}[Minimal Transportation Cost]
\label{def:minimaltransportationcost}
Let us take a cost function $c:X \times Y \to \erreinf$ as in Definition \ref{def:transp_funct}. The minimal transportation cost functional $W_c:\PP(X)\times\PP(Y)\to\erreinf$ is defined as  
\begin{equation}
    \label{eq:transportcost}
    (\mu,\nu)\to W_c(\mu,\nu):=\inf_{\pi\in \Pimn} \T(\pi).
\end{equation}
The value $W_c\mn$ is also called Wasserstein cost between $\mu$ and $\nu$.
\end{definition}

\noindent By making further assumptions on $c$, it is possible to prove that the infimum in \eqref{eq:transportcost} is a minimum. In particular, when the cost function is non negative, a minimizing solution exists. We denote with $\Gamma_o\mn$ the set of minimizers. For a complete discussion on the existence of the solution, we refer to \cite[Chapter 4]{Villani2008}.

\begin{lemma}[Measurable Selection of Plans, Villani \cite{Villani2008}, Chapter 5, Corollary 5.22]
\label{lm:measselplans}
Let $X$ and $Y$ be two Polish spaces and $c:X\times Y\to\erreinf$ a continuous cost function such that $\inf c >-\infty$. Given $\Omega$ a Polish space and $\lambda\in\PP(\Omega)$, consider a measurable map
\[
\omega \to (\mu_\omega,\nu_\omega)
\]
that goes from $\Omega$ to $\PP(X)\times\PP(Y)$. Then there is a measurable choice
\[
\omega\to\pi_\omega
\]
where for each $\omega$, $\pi_\omega$ is the optimal transportation plan between $\mu_\omega$ and $\nu_\omega.$
\end{lemma}

\subsubsection{One Dimensional Case}

\noindent When both the measures $\mu$ and $\nu$ are supported on $\erre$ and the cost function $c$ is convex, the solution exists, is unique, and is characterized by the pseudo-inverse function of $\mu$ and $\nu$.

\begin{definition}
Given $\mu,\nu\in\PP(\erre)$, the co-monotone transportation plan $\gamma_{mon}$ between $\mu$ and $\nu$ is defined as 
\[
\gamma_{mon}:=(F^{[-1]}_{\mu},F^{[-1]}_{\nu})_\#\Leb_{|[0,1]},
\]
where $F^{[-1]}_{\mu}$ and $F^{[-1]}_{\nu}$ are the pseudo-inverse of the cumulative functions of $\mu$ and $\nu$, respectively, and $\Leb_{|[0,1]}$ is the Lebesgue measure restricted on $[0,1]$.
\end{definition}

\begin{theorem}[Optimality of the co-monotone plan, Santambrogio \cite{Santambrogio2015}, Chapter 2, Theorem 2.9]
\label{th:optmonotoneplan}
Let $h:\erre\to\erre_+$ be a strictly convex function such that $h(0)=0$ and $\mu,\nu\in\PP(\erre)$. Consider the cost 
\[
c(x,y)=h(|x-y|)
\]
and suppose that this cost is feasible for the transportation problem. Then the Optimal Transportation problem has a unique solution which is $\gamma_{mon}$. 
\end{theorem}

\noindent Knowing how to express the optimal transportation plan, allow us to express the Wasserstein cost through the pseudo-inverse of the cumulative functions of $\mu$ and $\nu$.

\begin{corollary}[Santambrogio \cite{Santambrogio2015}, Chapter 2, Proposition 2.17]
\label{cor:structureW_c}
Let us take $\mu,\nu\in\PP(\erre)$. If $c(x,y)=(|x-y|)^p,$ with $p\geq1$, then 
\[
W_p^p(\mu,\nu)=\int_{[0,1]}(|F^{[-1]}_\mu-F^{[-1]}_\nu|)^pd\Leb.
\]
Moreover, for $p=1$,
\[
W_1(\mu,\nu)=\int_{\erre}|F_\mu(t)-F_\nu(t)|dt.
\]
\end{corollary}


\subsubsection{Wasserstein Distance}

\label{sc:Wassdist}

\noindent If we take $X=Y$ and choose $d$ as the cost function, the optimal transportation problem lifts the distance $d$ over the space $P_p(X)$. The resulting distance is called the Wasserstein distance.

\begin{definition}[Wasserstein Distance]
\label{def:Wasdist}
Let $(X,d)$ be a Polish space and $p\in [1,\infty)$. The $p-$order Wasserstein distance between the probability measures $\mu$ and $\nu$ on $X$ is defined as
\[
W_{d^p}^p(\mu,\nu):=\inf_{\pi \in \Pimn}\mathbb{T}_{d^p}(\pi).
\]
When $p=1$, the $1-$Wasserstein distance is also known as Kantorovich-Rubistein distance.
\end{definition}

\noindent The Wasserstein distance $W_{d^p}^p$ is well defined when used to compare measures in $P_p(X)$.

\begin{theorem}
\label{thm:Wpdistance}
The $W_{d^p}$ distance is a finite distance over $P_p(X)$. When the set $X$ is bounded, the $W_{d^p}$ distance induces the weak topology on the space $P_p(X)$. Moreover, $(P_p(X),W_{d^p})$ is a Polish space.
\end{theorem}
\section{Our Contribution}

In this section, we report our main results.  In paragraph \ref{ch1sect1}, we study the properties of the cardinal flows and introduce the pivot measure. Our main result on the structure of the Wasserstein cost is stated in Theorem \ref{th::conditionaltransport}.

\noindent In paragraph \ref{par:indip}, we show how to retrieve an optimal transportation plan from an optimal cardinal flow. Moreover, we prove that there exists an optimal transportation plan whose first and last marginal laws are independent given the other two.

\noindent Finally, in paragraph \ref{par_fin}, we analyze our formula in two specific frameworks.

\subsection{The Cardinal Flow and the Pivot Measure Formulation}
\label{ch1sect1}
From now on, we assume $X$ and $Y$ to be the product of smaller Polish spaces, i.e. $X=X_1\times X_2$ and $Y=Y_1\times Y_2$. In this framework, we can introduce the separable cost function and reformulate the optimal transportation problem as an optimal cardinal flow problem.

\begin{definition}[Separable Cost Function]
\label{def:separablecostfunction}
Let $X=X_1\times X_2$ and $Y=Y_1\times Y_2$ be two Polish spaces. We say that $c:X\times Y \to\erre$ is  \emph{separable} if there exist a pair of functions $c_1:X_1\times Y_1\to\erre$ and $c_2:X_2\times Y_2\to \erre$ such that
\[
c(\x,\y):=c_1(x_1,y_1)+c_2(x_2,y_2)
\]
for each $\x=(x_1,x_2)\in X$ and for each $\y=(y_1,y_2)\in Y$.
\end{definition}

\begin{definition}[Cardinal Flow]
 Let us take $\mu \in \probX$ and $\nu \in \probY$. We say that the couple of measures $\ff \in \PP(\Xfu)\times \PP(\Xfd)$ is a \emph{cardinal flow} between $\mu$ and $\nu$ if it satisfies the following conditions 
 \begin{itemize}
     \item The marginal on $X$ of $\funo$ is equal to $\mu$, i.e.
     \begin{equation}
     \label{def:cond1}
       \mu=\projX \funo. 
     \end{equation}
     \item The marginal on $Y$ of $\fdue$ is equal to $\nu$, i.e.
     \begin{equation}
     \label{def:cond2}
       \nu=\projY \fdue.
     \end{equation}
     \item The flows $\funo$ and $\fdue$ have the same marginal on $Y_1\times X_2$, i.e.
     \begin{equation}
         \label{def:cond3}
           \projxy\funo = \projxy\fdue.
     \end{equation}
   
 \end{itemize}
 We call the measures $\funo$ and $\fdue$ first and second cardinal flow, respectively.
 We denote with $\FFmn$ the set of all cardinal flows between $\mu$ and $\nu$.
\end{definition}

\begin{remark}
For any couple of probability measures $\mu$ and $\nu$, the set $\FFmn$ is nonempty. In fact, the couple $\ff$, defined as
\[
f^{(1)}=\mu \otimes \nu_1
\quad\quad\text{and}\quad\quad f^{(2)}=\mu_2\otimes \nu,
\]
is an element of $\FFmn$.
\end{remark}

\begin{remark}
The sets $\FFmn$ and $\FF(\nu,\mu)$ are, in general, not equal. For instance, let us take $X=Y=\erre^2$, $\mu=\delta_{(0,0)}$ and $\nu=\delta_{(1,1)}$. In this case, $\FFmn=\{\delta_{((0,0);1)}\}$ while $\FF(\nu,\mu)=\{\delta_{((1,1);0)}\}$. 
\end{remark}

\begin{definition}[Cardinal Flow Functional]
Given two probability measures $\mu \in \probX$, $\nu \in \probY$, and a separable cost function $c=c_1+c_2$ over $X \times Y$. We define the first and second cardinal transportation functionals as
\[
\ST^{(1)}(\funo)=\intXy c_1 d\funo\quad\quad\text{and}\quad\quad \ST^{(2)}(\fdue)=\intxY c_2 d\fdue
\]
where $F=\ff \in \FFmn$. The total cardinal flow functional is then defined as 
\[
\ST(F)=\ST^{(1)}(\funo)+\ST^{(2)}(\fdue).
\]
\end{definition}

\begin{theorem}
\label{th:cardflows}
Let $\mu \in \probX$, $\nu \in \probY$, and let $c:X \times Y \to \rpos$ be a separable cost function, then 
\[
\infpi\T(\pi) =\infcf  \ST\ff.
\]
\end{theorem}

\begin{proof}
The proof of this Theorem, in the discrete case, has been proposed in \cite{Auricchio2018}. The proof in our generic setting follows from similar arguments.
\end{proof}

\noindent Given $\mu\in\PP(X)$ and $\nu\in\PP(Y)$, let us consider the function $L: \Pimn \to \FFmn$ defined as
\begin{equation}
\label{def:Lop}
    L(\pi)=(\projXy(\pi),\projxY(\pi)).
\end{equation}
By the chain rule \ref{lm:chainrule}, it is easy to see that $L(\pi) \in \FFmn$ for each $\pi\in\Pimn$. Vice versa, given $\ff\in\FFmn$, by using the Gluing Lemma (Lemma \ref{lm:gluing}), we find $\pi\in\Pimn$ such that $L(\pi)=\ff$. This proves the identity $\FFmn=L(\Pimn)$. The function $L$ allows us to relate $\T$ and $\ST$, as it follows
\[
\T(\pi)=\ST(L(\pi)), \qquad \qquad \forall \pi \in \Pimn,
\]
which, in conjunction with the identity $L(\Pimn) =\FFmn$, allows us to conclude that the infimum of $\ST$ is actually a minimum and that the set of minimizers of $\ST$ coincides with the image of $\Gamma_o(\mu,\nu)$ through $L$. In particular, the cardinal flow problem inherits the uniqueness of the solution from the Optimal Transportation problem.

\begin{corollary}
If the optimal transportation plan is unique, so is the optimal cardinal flow.
\end{corollary}

\begin{remark}
\label{rmk:surjL}
Since the operator $L$ is only surjective and not injective, the reverse implication is not true, i.e., given an optimal cardinal flow $\ff$, there might exist a plethora of optimal transportation plans such that $L(\pi)=\ff$.
\end{remark}


\begin{definition}[Intermedium Measures]
\label{def:intermediate_measures}
Let $\mu \in \probX$ and $\nu \in \probY$. We define the set of intermedium measures between $\mu$ and $\nu$ as
\[
\II(\mu,\nu):=\bigg\{ \lambda \in \PP(Y_1\times X_2) \text{ s.t. } \projxx(\lambda)=\mu_2 \text{ and }\projyy(\lambda)=\nu_1 \bigg\}.
\]
\end{definition}

\begin{definition}
Given $\lambda \in \IImn$, we say that the cardinal flow $\ff \in \FFmn$ glues on $\lambda$ if
\[
\projxy \funo = \projxy \fdue =\lambda.
\]
\end{definition}


\begin{definition}[Pivot Measure]
Let us take $\mu\in \probX$, $\nu \in \probY$, and $c$ a separable cost function. We say that $\zeta \in \PP(Y_1\times X_2)$ is a \emph{pivot measure} between $\mu$ and $\nu$ if there exists an optimal cardinal flow $\ff$ that glues on it.
\end{definition}

\begin{remark}
\label{rmk:pivotisinter}
From Lemma \ref{lm:chainrule}, we have that all the pivot measures are also intermediate measures.
\end{remark}

\begin{theorem}
\label{th::conditionaltransport}
Let us take $\mu\in\PP(X)$, $\nu\in\PP(Y)$, and $c=c_1+c_2$ a separable cost function. For any pivot measure $\zeta$, it holds true the formula
\begin{equation}
\label{frm::Wcdecomposition}
W_c(\mu,\nu)=\int_{X_2}W_{c_1}(\mu_{|x_2},\zeta_{|x_2})d\mu_2+\int_{Y_1}W_{c_2}(\zeta_{|y_1},\nu_{|y_1})d\nu_1.
\end{equation}
\end{theorem}

\begin{proof}
Let $\zeta$ be a pivot measure between $\mu$ and $\nu$. From Remark \ref{rmk:pivotisinter} we know that $\zeta \in\IImn$. The Disintegration Theorem \ref{th:disintegration} allows us to write
\begin{equation}
\label{eq:disinthZ1}
  \zeta=\sigu\otimes \mu_2  \quad \quad \text{and} \quad \quad
   \zeta=\sigd \otimes \nu_1. 
\end{equation}
Similarly, we decompose $\mu$ and $\nu$ as
\begin{equation}
\label{eq:disinthZ3}
   \mu=\mu_{|x_2}\otimes\mu_2 \quad \quad \text{and} \quad \quad  \nu=\nu_{|y_1}\otimes\nu_1  
\end{equation}
respectively.
For $\mu_2-$almost every $x_2\in X_2$ is then well defined the problem
\begin{equation*}
    W_{c_1}(\mu_{|x_2},\sigu)=\inf_{\pi_{|x_2}\in \Pi(\mu_{|x_2},\sigu)}\int_{X_1\times Y_1}c_1d\pi_{|x_2}.
\end{equation*}

\noindent Theorem \ref{th:disintegration} assures us that the selections $x_2\to\mu_{|x_2}$ and $x_2\to \sigu$ are both measurable, hence, according to Lemma \ref{lm:measselplans}, there exists a measurable selection of optimal plans $\pi_{|x_2}$ for which holds true
\begin{equation}
    \label{eq:W1deq1}
    W_{c_1}(\mu_{|x_2},\sigu)=\int_{X_1\times Y_1}c_1d\pi_{|x_2}
\end{equation}
\noindent $\mu_2-$almost everywhere $x_2\in X_2$.
Similarly, there exists a measurable selection $\pi_{|y_1}$ for which, for $\nu_1-$almost every $y_1\in Y_1$, holds true
\begin{equation}
\label{eq:W1deq2}
W_{c_2}(\sigd,\nu_{|y_1})=\int_{X_2\times Y_2}c_2d\pi_{|y_1}.
\end{equation}
\noindent Let us now consider the measures $\funo\in \PP(X\times Y_1)$ and $\fdue\in \PP(X_2\times Y)$, defined as it follows
\begin{equation}
\label{eq:funoindip}
 \funo = \pi_{|x_2}\otimes\mu_2   \quad\quad\text{and}\quad\quad
    \fdue = \pi_{|y_1}\otimes\nu_1.
\end{equation}
The couple $\ff$ is a cardinal flow between $\mu$ and $\nu$, in fact, given an $\phi \in L^1_\mu$, we have
\begin{align*}
\int_{X}\phi \;d\mu&=\int_{X_2}\bigg(\int_{X_1\times Y_1}\phi\; d\mu_{|x_2} \bigg)\;d \mu_2\\
                &= \int_{X_2}\bigg(\int_{X_1\times Y_1}\phi \;d(\pr{X_1}_\#\pi_{|x_2}) \bigg)\;d \mu_2 \\
                &=\int_{X_2} \bigg(\int_{X_1\times Y_1}\phi\circ(\pr{X_1},Id_{X_2}))\;d\pi_{|x_2} \bigg)\;d \mu_2\\
                &= \int_{X\times Y_1}\phi\circ\pr{X} \;d \funo \\
                &= \int_{X}\phi\; d (\pr{X}_\# \funo) ,
\end{align*}
hence, $\pr{X}_\# \funo=\mu$. Similarly, we get
\begin{equation*}
    \pr{Y_1\times X_2}_\# \funo =\zeta,\quad\quad \pr{Y_1\times X_2}_\# \fdue=\zeta, \quad\quad \pr{Y}_\#\fdue=\nu,
\end{equation*}
%
hence $\ff \in \FFmn$. From the identities \eqref{eq:W1deq1} and \eqref{eq:W1deq2}, we have
\begin{align*}
    \int_{X_2}W_{c_1}(\mu_{|x_2},&\zeta_{|x_2})d\mu_2 + \int_{Y_1}W_{c_2}(\zeta_{|y_1},\nu_{|y_1})d\nu_1\\ &=\int_{X_2}\int_{X_1\times Y_1}c_1\;d\pi_{|x_2} d\mu_2+\int_{Y_1}\int_{X_2\times Y_2}c_2\;d\pi_{|y_1} d\nu_1\\
   &= \int_{X\times Y_1}c_1\;df^{(1)}+\int_{X_2\times Y}c_2\;df^{(2)},
\end{align*}
so that
\begin{align}
    \label{primadisug}
    \nonumber\int_{X_2}W_{c_1}(\mu_{|x_2},\zeta_{|x_2})d\mu_2 +\int_{Y_1}&W_{c_2}(\zeta_{|y_1},\nu_{|y_1})d\nu_1\\
    &\geq \min_{\ff} \ST\ff .
\end{align}

\noindent To prove the other inequality, let us take $\ff \in \FFmn$. By definition, we have
\[
\pr{X_2}_\# \funo =\pr{X_2}_\#\mu=\mu_2
\]
and
\[
\pr{Y_1}_\# \fdue =\pr{Y_1}_\#\nu=\nu_{1}.
\]
By disintegrating $\funo$ and $\fdue$ with respect to the variable $x_2$ and $y_1$, respectively, we find
\[
\funo=\psi_{|x_2}\otimes\mu_2 \quad \quad \text{and} \quad \quad\fdue=\phi_{|y_1}\otimes\nu_{1}.
\]
Let $\zeta$ be the measure on which $\funo$ and $\fdue$ glue, we have $\psi_{|x_2}\in\Pi(\mu_{|x_2},\zeta_{|x_2})$ $\mu_2-$a.e. and $\phi_{|y_1}\in\Pi(\zeta_{|y_1},\nu_{|y_1})$ $\nu_1-$a.e., indeed
\[
\mu=\pr{X}_\#\funo =\pr{X}_\#(\psi_{|x_2}\otimes\mu_2) =\big(\pr{X}_\#\psi_{|x_2}\big)\otimes\mu_2
\]
so that, by uniqueness of the conditional law, we have 
\[
\pr{X}_\#\psi_{|x_2}=\mu_{|x_2}\quad\quad \text{and}\quad \quad
\pr{Y_1\times X_2}_\#\psi_{|x_2}=\zeta_{|x_2},
\]
therefore, $\psi_{|x_2}\in\Pi(\mu_{|x_2},\zeta_{|x_2})$. Similarly, we find $\phi_{|y_1}\in\Pi(\zeta_{|y_1},\nu_{|y_1})$. Then, we have
\begin{eqnarray}
\label{thm::magg1}
\nonumber \int_{X\times Y_1}c^{(1)}\;d\funo&=&\int_{X_2}\int_{X_1\times  Y_1}c^{(1)}\;d\psi_{|x_2}d\mu_2\\
&\geq&\int_{X_2}W_{c_1}(\mu_{|x_2},\zeta_{|x_2})d\mu_2
\end{eqnarray}
and
\begin{eqnarray}
\label{thm::magg2}
\nonumber \int_{X_2\times Y}c^{(2)}\;d\fdue&=&\int_{Y_1}\int_{X_2\times  Y_2}c^{(2)}\;d\phi_{|y_1}d\nu_1\\
&\geq&\int_{Y_1}W_{c_2}(\zeta_{|y_1},\nu_{|y_1})d\nu_1.
\end{eqnarray}
By summing up the relations \eqref{thm::magg1} and \eqref{thm::magg2} we conclude
\[
W_c(\mu,\nu)\geq\int_{X_2}W_{c_1}(\mu_{|x_2},\zeta_{|x_2})d\mu_2+\int_{Y_1}W_{c_2}(\zeta_{|y_1},\nu_{|y_1})d\nu_1,
\]
which, along with \eqref{primadisug}, concludes the proof.
\end{proof}

 \noindent From formula \eqref{eq:disinthZ1}, we deduce that the computation of the Wasserstein cost between two generic measures can be achieved in two steps: detecting the pivotal measure and solving a family of lower dimensional problems. In particular, if we are able to solve the lower dimensional transportation problems, the only unknown left to determine is the pivot measure.

\begin{definition}[Pivoting functional]
Given $\mu\in\PP(X)$, $\nu\in\PP(Y)$, and $c=c_1+c_2$ a separable cost function, we define the pivoting functional $\Z:\IImn \to \erre$ as
\begin{equation}
\label{Zfunctional}
\Z:\zeta \to \int_{X_2}W_{c_1}(\mu_{|x_2},\zeta_{|x_2})d\mu_2+\int_{Y_1}W_{c_2}(\zeta_{|y_1},\nu_{|y_1})d\nu_1.
\end{equation}
\end{definition}



\begin{lemma}
Let $\mu\in\PP(X)$, $\nu\in\PP(Y)$, and $\lambda \in \IImn$. Then, there exists $\ff\in\FFmn$ such that
\[
\prxy_\# \funo =\lambda = \prxy_\# \fdue.
\]
\end{lemma}

\begin{proof}
Let $\lambda \in \IImn$. By the disintegrating $\lambda$, we get
\[
\lambda =\lambda_{|x_2}\otimes \lambda_2=\lambda_{|x_2}\otimes\mu_2\quad \quad \text{and}\quad \quad\mu =\mu_{|x_2}\otimes\mu_2.
\]
We define $\funo\in \PPXy$ as
\[
\funo=\big(\lambda_{|x_2}\otimes\mu_{|x_2}\big)\otimes\mu_2.
\]
It is easy to see that $\projX \funo=\mu$ and $\projXy\funo=\lambda$. Similarly, we define $\fdue$ as
\[
\fdue = \big(\lambda_{|y_1}\otimes \nu_{|y_1} \big)\otimes\nu_1,
\]
so that
\[
\nu = \projY \fdue \quad \quad \text{and}\quad \quad \lambda = \projxy \fdue,
\]
hence $\ff\in\FFmn$.
\end{proof}

\begin{theorem}
Given $\mu,\nu\in\PP(\erre^2)$, and $c=c_1+c_2$ a separable cost function, it holds true that
\[
W_c(\mu,\nu)=\inf_{\zeta \in \IImn} \Z (\zeta).
\]
\end{theorem}

\begin{proof}
Since any pivot measures $\zeta$ is an element of $\IImn$, Theorem \ref{th::conditionaltransport}, assure us
\[
W_c(\mu,\nu)\geq\inf_{\zeta \in \IImn} \Z (\zeta).
\]
To conclude we just need to prove the other inequality. 

\noindent Let us fix $\zeta \in \IImn$. Following the steps of the proof of Theorem \ref{th::conditionaltransport}, we disintegrate $\mu$, $\zeta$, and $\nu$ (see \eqref{eq:disinthZ1}-\eqref{eq:disinthZ3}), find the optimal transportation plans between the conditional measures, and define the cardinal flow as done in \eqref{eq:funoindip}. Since the couple $\ff \in \FFmn$, we have
\begin{eqnarray*}
\Z(\zeta)=\ST\ff\geq \inf_{\ff \in \FFmn} \ST\ff,
\end{eqnarray*}
hence
\[
\infIp \Z(\zeta) \geq \inf_{\ff \in \FFmn}\ST\ff=W_c(\mu,\nu),
\]
which concludes the proof.
\end{proof}

\subsection{Independence of the Optimal Coupling}
\label{par:indip}
As we noticed in Remark \ref{rmk:surjL}, from an optimal transportation plan we determine one optimal cardinal flow, however the vice versa is not true. The next example showcases how, even if we have a unique pivot measure and a unique optimal cardinal flow, we might retrieve an infinity of optimal transportation plans.

\begin{example}
\begin{figure}
  
  
  
\subfloat[]{
	\begin{minipage}[c][1\width]{
	   0.24\textwidth}
	   \centering
	   \includegraphics[width=0.5\textwidth]{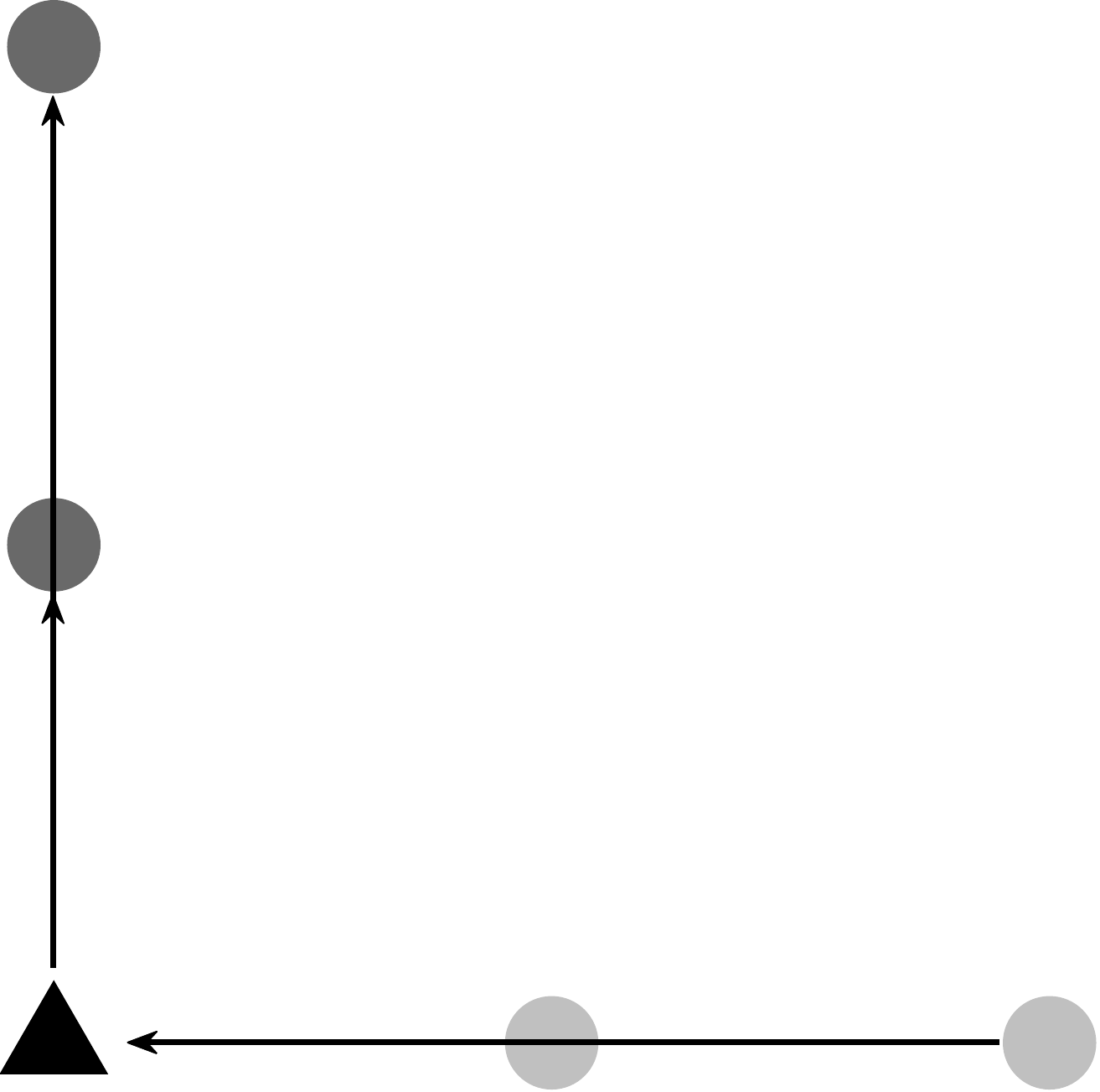}
	\end{minipage}}
\subfloat[]{
	\begin{minipage}[c][1\width]{
	   0.24\textwidth}
	   \centering
	   \includegraphics[width=0.5\textwidth]{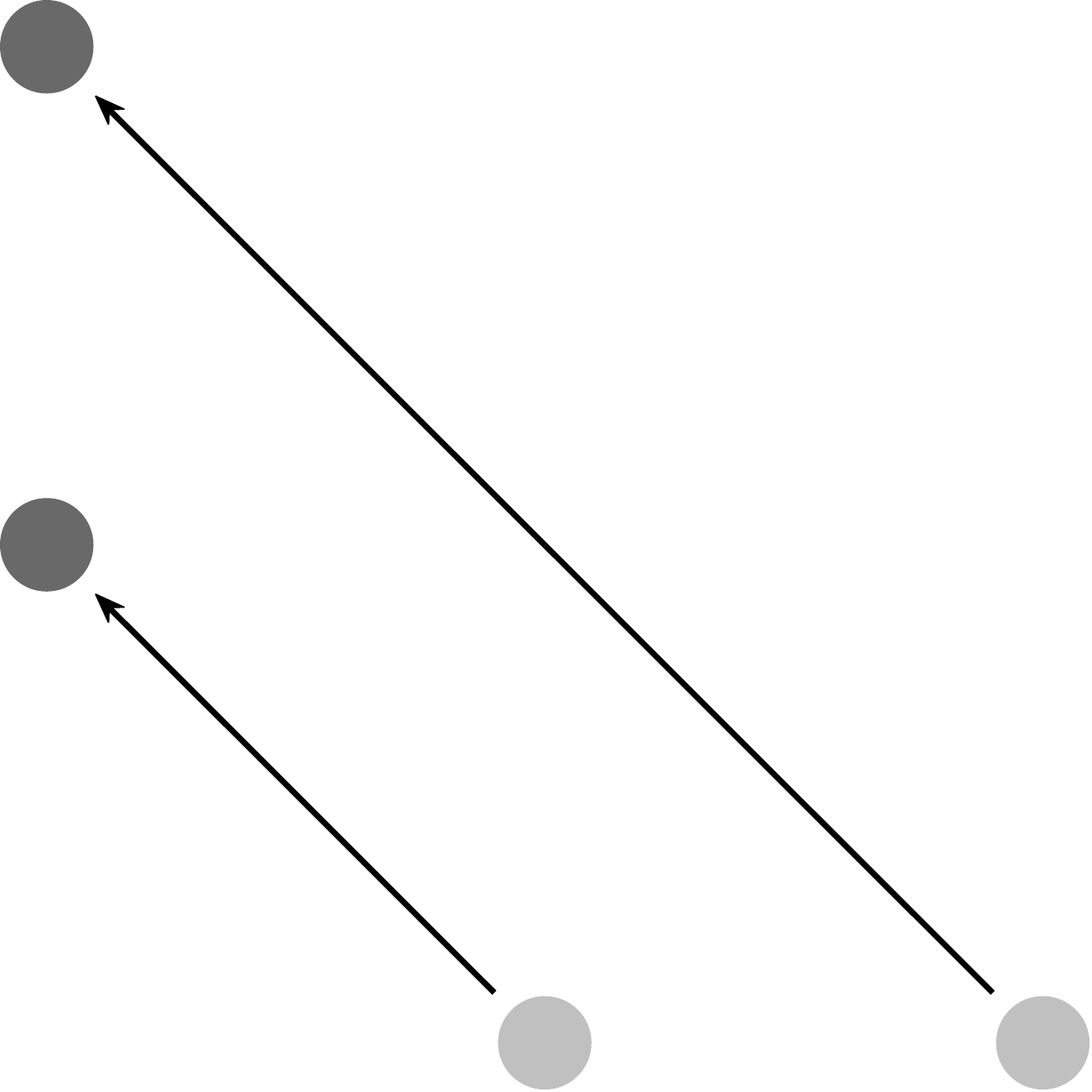}
	\end{minipage}}
  \subfloat[]{
	\begin{minipage}[c][1\width]{
	   0.24\textwidth}
	   \centering
	   \includegraphics[width=0.5\textwidth]{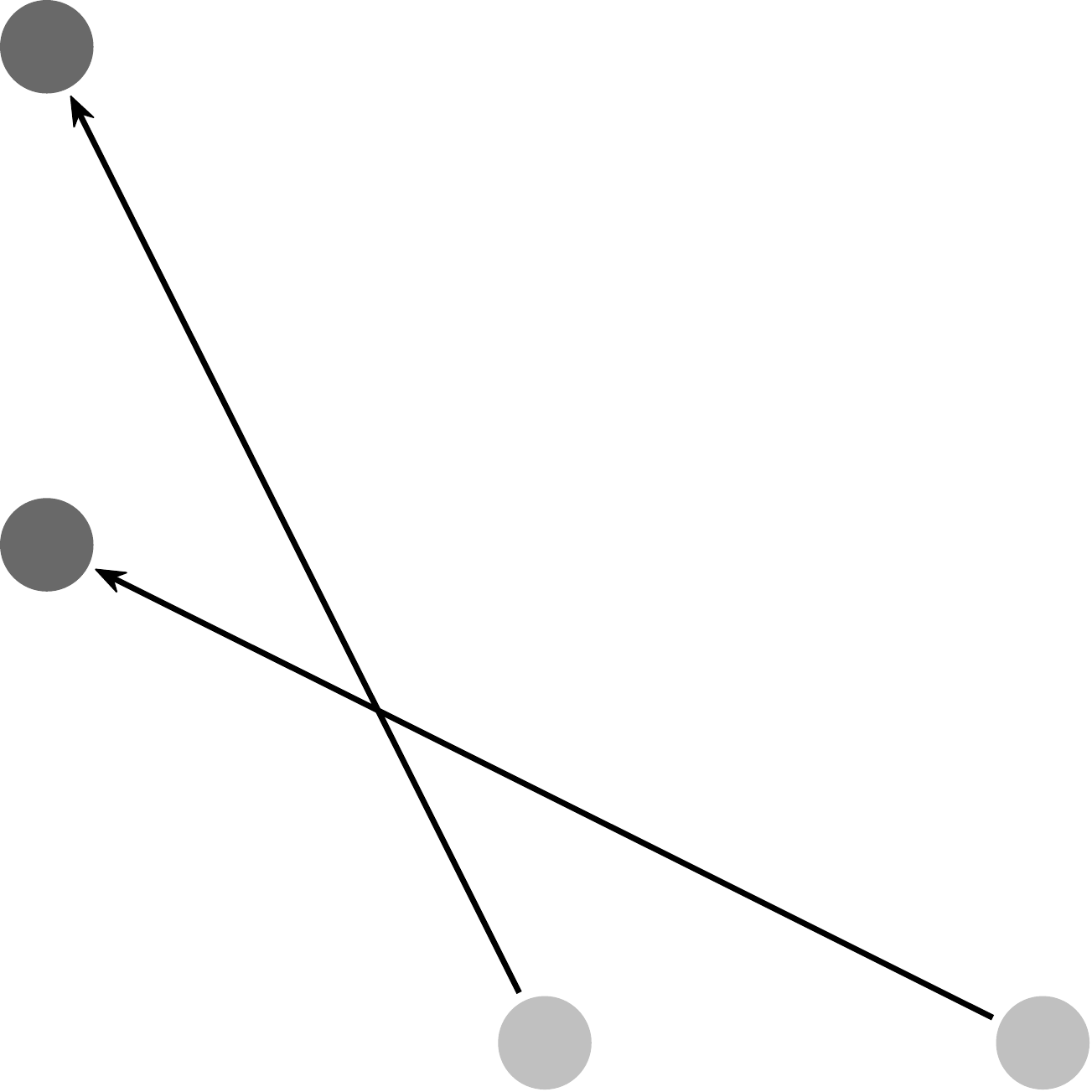}
	\end{minipage}}
  \subfloat[]{
	\begin{minipage}[c][1\width]{
	   0.24\textwidth}
	   \centering
	   \includegraphics[width=0.5\textwidth]{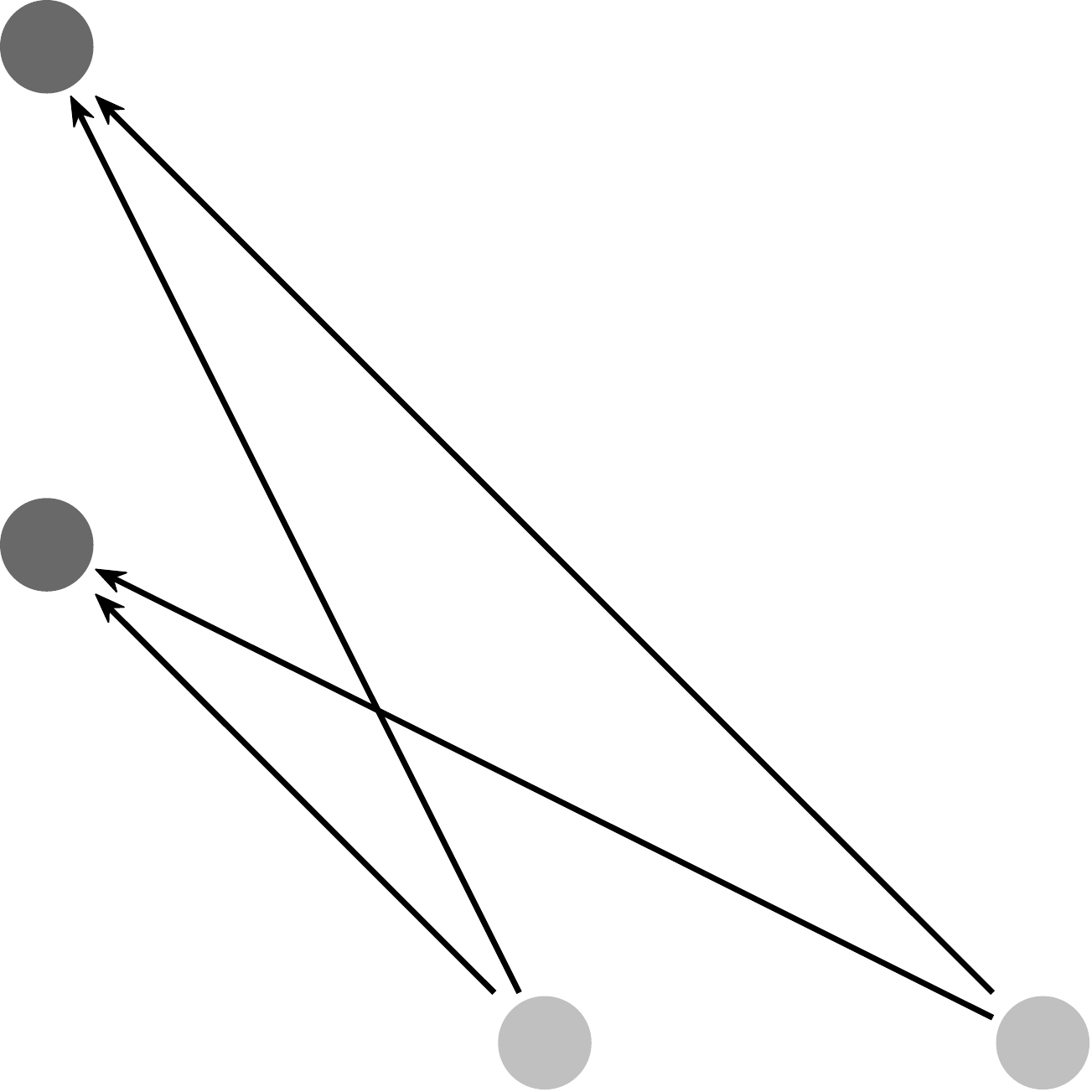}
	\end{minipage}}
\caption{The lack of uniqueness showcased in Example \ref{ex:switching}. The support of $\mu$ is indicated by light grey circles, the support of $\nu$ by dark grey circles. In Figure (a), we showcase the optimal cardinal flow. The support of the pivot measure $\zeta$ is indicated by the black triangle. In Figure (b), (c), and (d) are showcased the transportation plans $\pi_1$, $\pi_2$, and $\pi$, respectively. Every of those transportation plans induces the cardinal flow described in (a). }
\end{figure}

\label{ex:switching}
Let us take two probability measures on $\erre^2$, $\mu$ and $\nu$, defined as
\[
\mu:=\dfrac{1}{2}\bigg( \delta_{(1,0)}+ \delta_{(2,0)} \bigg)\quad \quad \text{and}\quad \quad\nu:=\dfrac{1}{2}\bigg( \delta_{(0,1)}+ \delta_{(0,2)}\bigg),
\]
where $\delta_{(x_1,x_1)}$ is the Dirac delta centered in $(x_1,x_2)\in\erre^2$. Since $\IImn=\{\delta_{(0,0)}\}$, the only possible pivot measure is $\zeta =\delta_{(0,0)}$, hence, the only (and therefore optimal) cardinal flow is
\[
\funo := \dfrac{1}{2}\bigg( \delta_{((1,0);0)} + \delta_{((2,0);0)}  \bigg)\quad\quad \text{and}\quad \quad\fdue := \dfrac{1}{2}\bigg( \delta_{((0,0);1)} + \delta_{((0,0);2)} \bigg).
\]
The measures $\pi,\pi_1$, and $\pi_2$, defined as
\begin{align*}
    &\pi:= \dfrac{1}{4}\bigg( \delta_{((1,0);(0,1))}+ \delta_{((1,0);(0,2))}+ \delta_{((2,0);(0,1))}+ \delta_{((2,0);(0,2))} \bigg),\\
    &\pi_1:=\dfrac{1}{2}\bigg(\delta_{((1,0);(0,1))}+\delta_{((2,0);(0,2))}  \bigg),\quad\;\text{and}\quad \;\pi_2:=\dfrac{1}{2}\bigg( \delta_{((1,0);(0,2))}+ \delta_{((2,0);(0,1))} \bigg)
\end{align*}
are all optimal transportation plans between $\mu$ and $\nu$, since they induce the same cardinal flow $\ff$. Moreover, any convex combination of $\pi_1$ and $\pi_2$ is an optimal transportation plan between $\mu$ and $\nu$.

\noindent The lack of uniqueness is due to a natural lack of memory. Roughly speaking, once the first cardinal flow $\funo$ allocates the mass into $(0,0)$, the mass coming from $(1,0)$ and $(2,0)$ merges in one point and loses its identity.
Therefore, when the second cardinal flow $\fdue$ reallocates the mass in $(0,0)$ and moves it vertically to complete the transportation we are unable to tell how much of the mass that ended in $(0,1)$, came from the point $(1,0)$ or $(2,0)$. The plans $\pi$, $\pi_1$, and $\pi_2$ are different for this reason: for $\pi$ just half of the mass in $(1,0)$ goes to $(0,1)$, for $\pi_1$ all the mass in $(1,0)$ goes to $(0,1)$ and, for $\pi_2$, none of the mass in $(1,0)$ goes to $(0,1)$.
\end{example}

\begin{lemma}
\label{lm:conditional_law}
Let us take $\mu\in\PP(X)$, $\nu\in\PP(Y)$, and $c: X\times Y \to [0,\infty)$ a separable cost function. Given $(\funo,\fdue)$ an optimal cardinal flow and $\zeta$ the pivotal measure related to $(\funo,\fdue)$, let $\funo_{|(x_2,y_1)}$ and $\fdue_{|(x_2,y_1)}$ be the conditional laws of $\funo$ and $\fdue$ given $\zeta$. Then, any measurable family of probability measures $\gamma_{(x_2,y_1)}$ satisfying
\begin{equation}
\label{eq:piopt}
    \gamma_{(x_2,y_1)}\in \Pi(\funo_{|(x_2,y_1)},\fdue_{|(x_2,y_1)}), 
\end{equation}
is such that
\[
\gamma_{(x_2,y_1)}\otimes \zeta\in \Gamma_o(\mu,\nu).
\]
In particular, the transportation plan
\begin{equation}
\label{lm:piopt}
\pi:=  \big( \funo_{|(x_2,y_1)}\otimes\fdue_{|(x_2,y_1)} \big)  \otimes \zeta
\end{equation}
minimizes $\T$.
\end{lemma}

\begin{proof}
Let $\pi\in\Pimn$ be defined as \ref{eq:piopt}. Since $L(\pi)=\ff$ and $\ff$ is optimal, we conclude $\pi\in\Gamma_o\mn$.
\end{proof}

\noindent As a straightforward consequence we get the following.

\begin{theorem}
Let $\mu\in\PP(X)$, $\nu\in\PP(Y)$, and $c=c_1+c_2$ a separable cost function. Let us assume the transportation problem between $\mu$ and $\nu$ has a unique solution $\pi$. Then, if $(X_1,X_2,Y_1,Y_2)$ is the coupling inducing the law $\pi$, we have
\begin{enumerate}
    \item \label{pt_1} $X_1$ and $Y_2$ are conditionally independent given $X_2$ and $Y_1$,
    \item \label{pt_2} $X_2$ and $Y_1$ are conditionally independent given $X_1$ and $Y_2$.
\end{enumerate}
\end{theorem}

\begin{proof}
Let $\pi$ be the optimal transportation plan, $\ff=L(\pi)$, and $\zeta$ the pivot measure. Since the plan \eqref{lm:piopt} is optimal, we have
\[
\pi=\big( \funo_{|(x_2,y_1)}\otimes\fdue_{|(x_2,y_1)} \big)  \otimes \zeta,
\]
by the uniqueness of the disintegration we find $\pi_{|(x_2,y_1)}=\funo_{|(x_2,y_1)}\otimes\fdue_{|(x_2,y_1)}$, which proves \ref{pt_1}. By swapping $\mu$ and $\nu$, we conclude \ref{pt_2}.
\end{proof}

\begin{corollary}
\label{lm:pivotdelta}
Let $\mu\in\PP(X)$, $\nu\in \PP(Y)$. If there exist $\bar{x}_2\in X_2$ for which
\begin{equation}
    \label{eq:casopar1}
    \mu_2(\{\bar{x}_2\})=\mu(X_1\times \{\bar{x}_2\})=1
\end{equation}
the pivot measure is $\zeta= \nu_1\otimes\delta_{\bar{x}_2}$. Similarly, if there exists $\bar{y}_1$ such that
\begin{equation}
    \label{eq:casopar2}
    \nu_1(\{\bar{y}_1\})=\nu(\{\bar{y}_1\}\times Y_2)=1,
\end{equation}
the pivot measure is $\zeta=\delta_{\bar{y}_1}\otimes\mu_2$. Moreover, if there exists a couple $(\bar{x}_2,\bar{y}_1)$ satisfying both \eqref{eq:casopar1} and \eqref{eq:casopar2}, then $\zeta = \delta_{(\bar{y}_1,\bar{x}_2)}$ is the pivot measure and each $\pi\in\Pimn$ is optimal.
\end{corollary}

\begin{proof}
The first two statements follow from the fact that $\IImn$ contains only one measure, which is $\nu_1\otimes \delta_{\bar{x}_2}$ in the first case and $\delta_{\bar{y}_1}\otimes \mu_2$ in the second one. If both \eqref{eq:casopar1} and \eqref{eq:casopar2} hold, also $\FFmn$ contains only one element, $\ff=(\mu\otimes\delta_{\bar{y}_1},\delta_{\bar{x}_2}\otimes \nu)$. In particular, $L(\pi)=\ff$ for each $\pi\in\Pimn$, therefore each $\pi \in \Pimn$ is optimal.
\end{proof}

\subsection{Two Specific Frameworks}
\label{par_fin}
To conclude, we inhabit our studies in two specific frameworks. In the first one, the cost function is a separable distance, i.e. the sum of two distances. In the second one, the measures are supported over $\erre^2$ and the cost function has the form \eqref{introd_formula1}.

\subsubsection{Separable Distances}

When $X=Y$ and we choose $c=d$ as a cost function, Theorem \ref{thm:Wpdistance} states that the Optimal Transportation problem lifts the distance structure from $X$ to $P_p(X)$. When $d$ is separable, the induced distance $W_d$ inherits a weaker version of the separability, i.e.
\[
W_d\mn=W_d(\mu,\zeta)+W_d(\zeta,\nu),
\]
for any  pivot measure $\zeta$.\\

\noindent Since $X=Y$, we need to slightly change the notations in order to avoid confusion. We denote the generic point $(\x^{(1)},\x^{(2)})\in X \times X$ with
\[
(\x^{(1)},\x^{(2)})=\big((x_1^{(1)},x_2^{(1)}),(x_1^{(2)},x_2^{(2)})\big),
\]
hence, we denote with $\x^{(i)}$ the $i-$th component in the space $X\times X$ and we denote with $x^{(i)}_j$ the $j-$th component of $\x^{(i)}$. The projections $\pr{}^{(i)}:X \times X \to X$ and $p_j^{(i)}:X\times X\to X_j$ are defined as
\begin{equation}
    \pr{}^{(i)}(\x^{(1)},\x^{(2)}):=\x^{(i)}\quad \quad \text{and}\quad \quad    p_j^{(i)}(\x^{(1)},\x^{(2)}):=x^{(i)}_j
\end{equation}
for $i=1,2$ and $j=1,2$. In particular, we have
\[
\pr{}^{(i)}=(p_1^{(i)},p_2^{(i)}).
\]

\begin{theorem}
\label{th:sepWd}
Let us take $\mu,\nu \in \probX$, where $X=X_1\times X_2$, and let $d$ be a separable distance over $X$, i.e.
\[
d(\x,\y)=d_1(x_1,y_1)+d_2(x_2,y_2), \qquad \qquad \forall \x,\y \in X,
\]
where $d_1:X_1\times Y_1\to\erre$ and $d_2:X_2\times Y_2\to\erre$ are two distances. Then, for any pivot measure $\zeta \in \IImn$, we have
\begin{equation}
    W_d(\mu,\nu)=W_d(\mu,\zeta)+W_d(\zeta,\nu).
\end{equation}
\end{theorem}

\begin{proof}
Since $W_d$ is a distance over $\PP(X)$, by the triangular inequality, we have
\begin{equation}
\label{th:trineq}
 W_d(\mu,\nu)\leq W_d(\mu,\zeta)+W_d(\zeta,\nu),
\end{equation}
for any $\zeta\in\PP(X)$, and, in particular, for each $\zeta\in \IImn.$

\noindent To prove the other inequality, let us take $\ff\in \FFmn$ an optimal cardinal flow. We define
\[
\pi^{(1)}=\big(\pr{}^{(1)},(p^{(2)}_1,p^{(1)}_2)\big)_{\#}f^{(1)}
\]
and
\[
\pi^{(2)}=\big((p^{(2)}_1,p^{(1)}_2),\pr{}^{(2)}\big)_{\#}f^{(2)}.
\]

\noindent By definition, we have
\begin{equation}
\label{th:DSrel1}
    \pr{}^{(1)}\circ\big(\pr{}^{(1)},(p^{(2)}_1,p^{(1)}_2)\big)=(p^{(1)}_1,p^{(1)}_2)
\end{equation}
and
\begin{equation}
\label{th:DSrel2}
     \pr{}^{(2)}\circ\big(\pr{}^{(1)},(p^{(2)}_1,p^{(1)}_2)\big)=(p^{(2)}_1,p^{(1)}_2).
\end{equation}

\noindent Through the relations \eqref{th:DSrel1}, \eqref{th:DSrel2}, and the chain rule for the push-forwards (Lemma \ref{lm:chainrule}) we get
\begin{eqnarray*}
\pr{}^{(1)}_\# \pi^{(1)}&=&(p^{(1)}_1,p^{(1)}_2)_\# \Big(\big(\pr{}^{(1)},(p^{(2)}_1,p^{(1)}_2)\big)_{\#}\funo\Big)=(p^{(1)}_1,p^{(1)}_2)_\# \funo=\mu
\end{eqnarray*}
and
\[
\pr{}^{(2)}_\# \pi^{(1)}=\zeta,
\]
hence $\pi^{(1)} \in \Pi(\mu,\zeta).$ Similarly, we get $\pi^{(2)}\in \Pi(\zeta,\nu)$. Finally, by Theorem \ref{th:cardflows} we have that
\begin{eqnarray}
\nonumber W_d(\mu,\nu)&=&\int_{X\times X_1}d_1\;d\funo+\int_{X_2\times X}d_2\;d\fdue\\
\nonumber&=&\int_{X\times X_1}d\circ \big(\pr{}^{(1)},(p^{(2)}_1,p^{(1)}_2)\big) d\funo\\
\nonumber&\qquad&+\int_{X_2\times X}d\circ\big((p^{(2)}_1,p^{(1)}_2),\pr{}^{(2)}\big) d\fdue\\
\nonumber&=&\int_{X\times X}d\;d\pi^{(1)}+\int_{X\times X}d\;d\pi^{(2)}\\
\nonumber&\geq& W_d(\mu,\zeta)+W_d(\zeta,\nu).
\end{eqnarray}
The latter inequality, in conjunction with \eqref{th:trineq}, allows us to conclude
\[
W_d(\mu,\nu)=W_d(\mu,\zeta)+W_d(\zeta,\nu)
\]
for any pivot measure $\zeta$.
\end{proof}

\begin{remark}
The previous Theorem states that any pivot measure minimizes the functional
\[
\Theta:\lambda \to W_d(\mu,\lambda)+W_d(\lambda,\nu).
\]
However, the reverse implication is not true. Let us consider, for
instance, $X=\erre^2$,
\[
\mu=\dfrac{1}{2}\big[\delta_{(0,0)}+\delta_{(7,1)}\big]\quad \quad \text{ and }\quad \quad \nu=\dfrac{1}{2}\big[\delta_{(1,1)}+\delta_{(8,0)}\big].
\]
It is easy to see that the only pivot measure is $\zeta=\dfrac{1}{2}\big[\delta_{(1,0)}+\delta_{(8,1)}\big]$. Since $\mu_2=\nu_2$, we have $\nu\in\IImn$ and, therefore
\begin{equation*}
    W_d(\mu,\nu)=\inf_{\lambda\in\IImn}W_d(\mu,\lambda)+W_d(\lambda,\nu)\leq W_d(\mu,\nu)+W_d(\nu,\nu)=W_d(\mu,\nu)
\end{equation*}
hence $\nu$ minimizes $\Theta$, although it is not the pivot measure.
\end{remark}

\subsubsection{Cardinal Flow in the Euclidean Space}
\noindent Let us now consider $X=Y=\erre^2$ and a separable cost function $c=c_1+c_2$, with
\begin{equation}
\label{formula:powerlikecosts}
c_i(x,y)=h(|x_i-y_i|),
\end{equation}
where $h:\erre \to \rpos$ is a convex function such that $h(0)=0$. Due to the structure of $c_i$,  Corollary \ref{cor:structureW_c} allows us to rewrite \eqref{frm::Wcdecomposition} in terms of pseudo-inverse functions.

\begin{theorem}
\label{thm:euclideanopt}
Let us take $\mu,\nu\in\PP(\erre^2)$, and $c=c_1+c_2$ a separable cost function, where the functions $c_i$ are as in \eqref{formula:powerlikecosts}. Assume that $\zeta \in \IImn$ is a pivot measure between $\mu$ and $\nu$, then it holds true
\begin{eqnarray*}
W_c(\mu,\nu)&=&\int_{\erre}\int_{[0,1]} h(|F^{[-1]}_{\mu_{|x_2}}(s)-F^{[-1]}_{\zeta_{|x_2}}(s)|)dsd\mu_2\\
&\quad&+\int_{\erre}\int_{[0,1]} h(|F^{[-1]}_{\zeta_{|y_1}}(t)-F^{[-1]}_{\nu_{|y_1}}(t)|)dtd\nu_1,
\end{eqnarray*}
where $F^{[-1]}_{\mu_{|x_2}},F^{[-1]}_{\zeta_{|x_2}},F^{[-1]}_{\zeta_{|y_1}}$, and $F^{[-1]}_{\nu_{|y_1}}$ are the pseudo-inverse functions of the cumulative function of $\mu_{|x_2},\zeta_{|x_2},\zeta_{|y_1}$, and $\nu_{|y_1}$, respectively. In particular, if $h=|\circ|$, we have
\[
W_1(\mu,\nu)=\int_\erre \int_\erre |F_{\mu_{|x_2}}(s)-F_{\zeta_{|x_2}}(s)|dsd\mu_2+\int_{\erre}\int_\erre |F_{\zeta_{|y_1}}(t)-F_{\nu_{|y_1}}(t)|dtd\nu_1.
\]
\end{theorem}

\noindent In this framework, we are able to express the optimal transportation plan as long as we detect the pivot measure $\zeta$. In particular, if the hypothesis of Corollary \ref{lm:pivotdelta} are satisfied, we retrieve an explicit formula for the optimal transportation plan between measures in $\erre^2$.

\begin{theorem}
\label{thm:last}
Let $\mu,\nu\in\PP(\erre^2)$, and let $c=c_1+c_2$ be a separable cost function where the functions $c_i$ are as in \eqref{formula:powerlikecosts}. If $\mu$ is such that
\[
\mu(\{x_2=0\})=1,
\]
then the optimal cardinal flow $\ff$ between $\mu$ and $\nu$ is defined as
\begin{equation}
    \label{eq:KR}
    \funo:=(F^{[-1]}_{\mu_1},F^{[-1]}_{\nu_1})_\#\Leb_{|[0,1]}\otimes \delta_0\quad\text{and}\quad\fdue:=(\delta_{0}\otimes (F^{[-1]}_{\nu_{|y_1}})_\#\Leb_{|[0,1]})\otimes \nu_1.
\end{equation}
In particular, we have
\begin{equation}
   \label{eq:evalutation}
W_c\mn=W_{c_1}(\mu_1,\nu_1)+\int_{\erre}W_{c_2}(\delta_{0},\nu_{|y_1})d\nu_1. 
\end{equation}
\end{theorem}

\begin{proof}
From Corollary \ref{lm:pivotdelta}, we have that the pivot measure between $\mu$ and $\nu$ is
\[
\zeta=\nu_1\otimes \delta_0.
\]
From Theorem \ref{thm:euclideanopt}, we get
\[
\funo=\gamma^{(1)}_{|x_2}\otimes \delta_{0}\quad \quad \text{and} \quad \quad \fdue=\gamma^{(2)}_{|y_1}\otimes \nu_1,
\]
where $\gamma^{(1)}_{|x_2}$ is the measurable family of optimal transportation plans between $\mu_{|x_2}=\mu_1$ and $\zeta_{|x_2}=\nu_1$. Similarly, $\gamma^{(2)}_{|y_1}$ is the measurable family of optimal transportation plans between $\zeta_{|y_1}=\delta_{0}$ and $\nu_{|y_1}$. Theorem \ref{th:optmonotoneplan} allows us to characterize the optimal cardinal flow as in \eqref{eq:KR}. 

\noindent Finally, by evaluating the functional $\ST$ in $\ff$, we find \eqref{eq:evalutation} and conclude the proof.
\end{proof}

\begin{remark}
One of the optimal transportation plans inducing the cardinal flow in \eqref{eq:KR} is the Knothe-Rosenblatt rearrangement, \cite{knothe1957,rosenblatt1952}.
\end{remark}

\begin{figure}[!t]
    \centering
    \includegraphics[width=\linewidth]{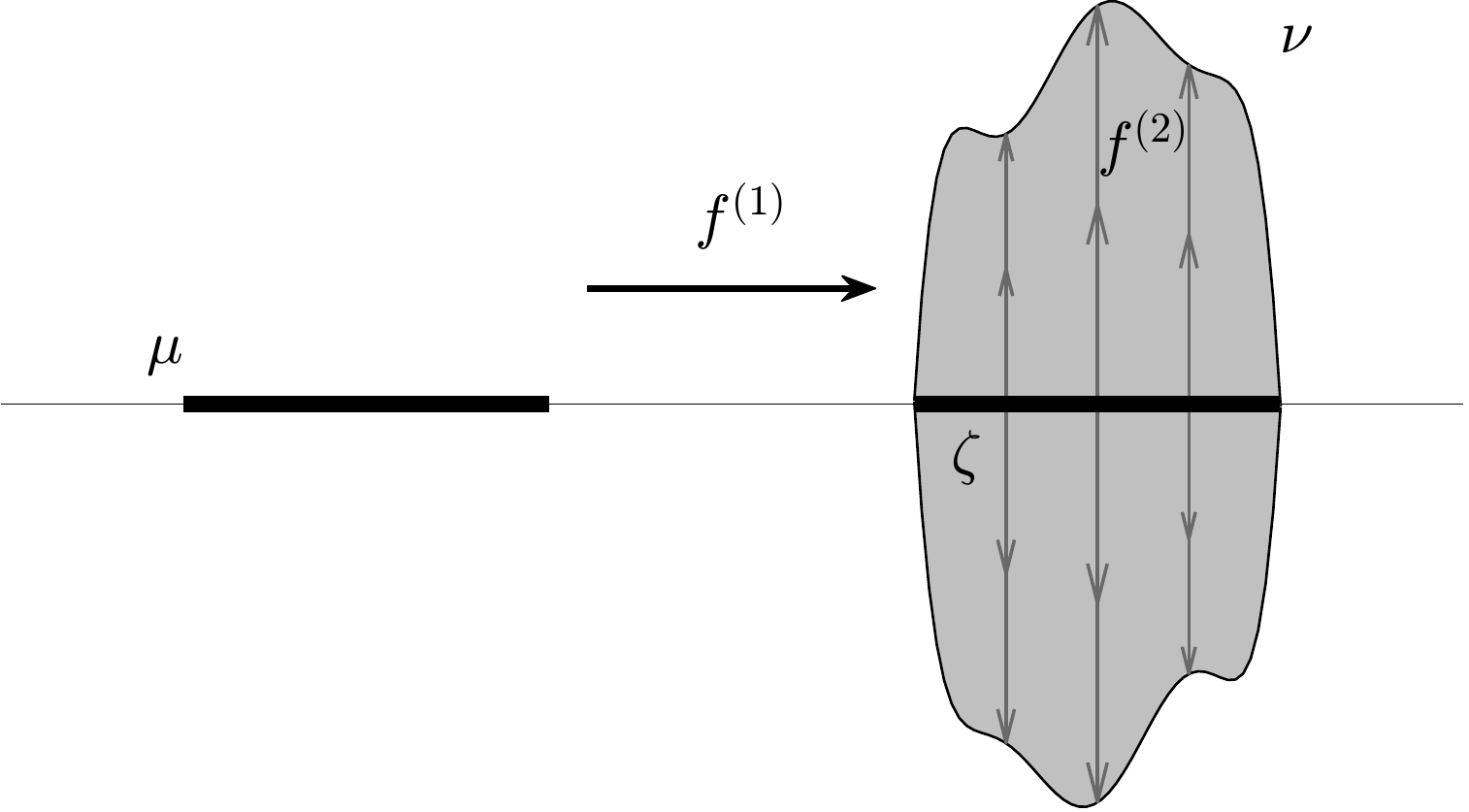}
    \caption{The cardinal flow found in Theorem \ref{thm:last}.}
    \label{fig:my_label}
\end{figure}

\noindent Theorem \ref{thm:last} becomes particularly useful when we take the squared Euclidean distance as cost function, i.e.
\[
c(\x,\y):=(|x_1-y_1|)^2+(|x_2-y_2|)^2.
\]
Indeed, since $c$ is invariant under isometries, Theorem \ref{thm:last} generalizes to any $\mu$ supported on a straight line.

\begin{corollary}
Let $\mu,\nu\in\PP(\erre^2)$, and $c(\x,\y):=(|x_1-y_1|)^2+(|x_2-y_2|)^2$. If there exist a triple $a,b,q\in\erre$ such that 
\[
\mu(\{ax_1+bx_2=q\})=1,
\]
then the optimal transportation plan between $\mu$ and $\nu$ is $\pi_O:=(O,O)_\#\pi$, where $O:\erre^2\to\erre^2$ is a rotation that sends the set $\{x_2=0\}$ in $\{ax_1+bx_2=q\}$ and $\pi$ is the optimal transportation plan between $(O^{(-1)})_\#\mu$ and $(O^{(-1)})_\#\nu$.
\end{corollary}

\begin{proof}
Since the Euclidean distance is invariant under isometries, we have
\[
\T(\pi)=\T((O^{(-1)},O^{(-1)})_\#\pi)
\]
for any rotation $O$ and, therefore,
\[
W_2^2((O^{(-1)})_\#\mu,(O^{(-1)})_\#\nu)=W_2^2\mn.
\]
Since $(O^{(-1)})_\#\mu$ satisfies the hypothesis of Theorem \ref{thm:last}, we conclude the thesis.
\end{proof}

\section*{Acknowledgments}
 We thank Stefano Gualandi and Marco Veneroni for their feedbacks. We also thank Gabriele Loli for enhancing the images of this paper.

\bibliographystyle{plain}
\bibliography{ref}
\end{document}